\documentclass[11pt]{amsart}

\usepackage{color}
\usepackage{graphicx}
\usepackage{epsfig}

\theoremstyle{plain}
\newtheorem{Thm}{Theorem}[section]
\newtheorem{Cor}[Thm]{Corollary}
\newtheorem{Lem}[Thm]{Lemma}

\newtheorem{Rem}[Thm]{\sl Remark}

\theoremstyle{definition}

\theoremstyle{remark}

\errorcontextlines=0 \numberwithin{equation}{section}

\begin{document}
	\title[Sets of Fractional Difference Sequences]
	{Compact Operators on the Sets of Fractional Difference Sequences}
	
	\author{Faruk \"{O}zger}
	
	\address{Department of Engineering Sciences, Izmir Katip Celebi University, Izmir, Turkey}
	\email[\"{O}zger]{faruk.ozger@ikc.edu.tr; farukozger@gmail.com}

	\keywords{Fractional difference sequence spaces, compact operators, Hausdorff measure of noncompactness \\
		This work is a part of the research project supported by Izmir Katip Celebi University Scientific Research Project Coordination Unit}

\begin{abstract}
Fractional difference sequence spaces have been studied in the literature recently. In this work, some identities or estimates for the operator norms and the Hausdorff measures of noncompactness of certain operators on some difference sequence spaces of fractional orders are established. Some classes of compact operators on those spaces are characterized. The results of this work are more general and comprehensive then many other studies in literature.
\end{abstract}

\maketitle

\section{Introduction}
The most common type of sets of sequences are probably the sets of difference sequences among the sequence spaces studied. The difference sequence spaces first introduced in K\i zmaz's study \cite{p6}. Many authors have made efforts to investigate the topological structures of these spaces during the past decade (see \cite{AyBas},\cite{ema},\cite{fofb},\cite{fofb2},\cite{candan},\cite{kirisci},\cite{ColEt}). Compact operators on the sets of difference sequences have been characterized in (\cite{ema22},\cite{mo},\cite{MalRak},\cite{Muretal}). We also refer to (\cite{jarrah},\cite{Wil2},\cite{p15},\cite{mara},\cite{mov},\cite{Mara7},\cite{Stieglitz}) for further studies in theory of $FK$-spaces and its applications.
\\
More recently, certain difference sequence spaces of fractional orders have been introduced by Baliarsingh \cite{BalDut}. Certain Euler difference sequence spaces of fractional order and related dual properties have been studied by Kadak and Baliarsingh \cite{KadBal}. Topological properties of certain sequence spaces that are combined by the mean operator and the fractional difference operator are investigated by Furkan \cite{Fur}.

The rest of the paper is organized as follows. In the rest of this section, we consider fractional operators, their properties and fractional sets of sequences $c_{0}(\Delta ^{(  \tilde{\alpha} )} )$, $c(\Delta ^{(  \widetilde{\alpha} )} )$ and $\ell_{\infty} (\Delta ^{(  \widetilde{\alpha} )}  )$. In section 2, we will determine the $\beta$ duals of fractional sets of sequences and
characterize matrix transformations on them. We also examine operator norms of our spaces. In section 3, we will study on characterizations of some compact operators by applying Hausdorff measure of noncompactness.

\subsection{Fractional Difference Operators}

The gamma function of a real number x (except zero and the negative integers) is defined by an improper integral:

\begin{eqnarray*}
\Gamma \left ( x \right )=\int_{0}^{\infty}e^{-t}t^{x-1}dt.
\end{eqnarray*}
\\
It is known that for any natural number  $n$, $\Gamma(n+1)=n!$ and $\Gamma(n+1)=n\Gamma(n)$ holds for any real number $n\notin\left \{ 0,-1,-2,... \right \}$.

The fractional difference operator for a fraction $\tilde{\alpha}$ have been defined in \cite{BalDut} as

\begin{eqnarray}\label{del}
\Delta ^{\left (  \tilde{\alpha} \right )}(x_{k})=\sum_{i=0}^{\infty}\left ( -1 \right )^{i}\frac{\Gamma \left ( \tilde{\alpha} +1 \right )}{\Gamma \left ( \tilde{\alpha} -i+1 \right )}x_{k-i}.
\end{eqnarray}

It is assumed that the series defined in (\ref{del}) is convergent for $x\in\omega$. \\

Let $m$ be a positive integer then recall the difference operators $\Delta^{(1)}$ and $ \Delta ^{(m)}$ by :

\begin{eqnarray*}
(\Delta^{(1)}x)_{k}=\Delta^{(1)}x_{k}=x_{k}-x_{k-1}
\end{eqnarray*}
and
\begin{eqnarray*}
(\Delta^{(m)}x)_{k}=\sum_{i=0}^{m}(-1)^{i}\binom{m}{i}x_{k-i}.
\end{eqnarray*}

We write $\Delta$ and $\Delta^{(m)}$ for the matrices with $\Delta_{nk}=(\Delta^{(1)}e^{(k)})_{n}$ and
$\Delta_{nk}^{(m)}=(\Delta^{(m)}e^{(k)})_{n}$
for all $n$ and $k$.

We may write the fractional difference operator as an infinite matrix:
\begin{equation*}
\Delta ^{\left (  \tilde{\alpha} \right )}_{nk} =\left\{
\begin{array}{lll}
(-1)^{n-k}\frac{\Gamma (\tilde{\alpha}+1)}{(n-k)!\Gamma (\tilde{\alpha}-n+k+1)} &  & (0\leq k\leq n) \\
0 &  & (k>n).%
\end{array}%
\right.
\end{equation*}

\begin{Rem}
The inverse of fractional difference matrix is given by
\begin{equation*}
\Delta ^{\left (  -\tilde{\alpha} \right )}_{nk} =\left\{
\begin{array}{lll}
(-1)^{n-k}\frac{\Gamma (-\tilde{\alpha}+1)}{(n-k)!\Gamma (-\tilde{\alpha}-n+k+1)} &  & (0\leq k\leq n) \\
0 &  & (k>n).%
\end{array}%
\right.
\end{equation*}
\end{Rem}
For some values of $\tilde{\alpha}$, we have
\begin{eqnarray*}
\Delta ^{1/2}x_{k}&=&x_{k}-\frac{1}{2}x_{k-1}-\frac{1}{8}x_{k-2}-\frac{1}{16}x_{k-3}-\frac{5}{128}x_{k-4}-...\\
\Delta ^{-1/2}x_{k}&=&x_{k}+\frac{1}{2}x_{k-1}+\frac{3}{8}x_{k-2}+\frac{5}{16}x_{k-3}+\frac{35}{128}x_{k-4}+...\\
\Delta ^{2/3}x_{k}&=&x_{k}-\frac{2}{3}x_{k-1}-\frac{1}{9}x_{k-2}-\frac{4}{81}x_{k-3}-\frac{7}{243}x_{k-4}-...
\end{eqnarray*}
\begin{Thm}The following results hold:
\begin{itemize}
\item[(i)]$\Delta ^{\left (  \tilde{\alpha} \right )}\circ \Delta ^{\left ( - \tilde{\alpha} \right )}=I$, where $I$ is identity on $x\in\omega$.
\item[(ii)]$\Delta ^{\left (  \tilde{\alpha} \right )} \Delta ^{\left ( \tilde{\beta} \right )}=
\Delta ^{\left ( \tilde{\alpha}+ \tilde{\beta} \right )}$.
\item[(iii)]$\Delta ^{\left (  \tilde{\alpha} \right )}( \Delta ^{\left ( - \tilde{\alpha} \right )}x_{k})=x_{k}$.
\end{itemize}
\end{Thm}

\begin{proof}
Since the proofs of Parts (i) and (ii) can similarly be obtained, we only prove Part (iii).

\begin{eqnarray*}
\begin{array}{lll}
\Delta ^{\left (  \tilde{\alpha} \right )}( \Delta ^{\left ( - \tilde{\alpha} \right )}x_{k})  & = & \Delta ^{\left (  \tilde{\alpha} \right )}\left \{x_{k}+x_{k-1}\alpha +x_{k-2}\frac{\alpha (\alpha +1)}{2!}+x_{k-3}\frac{\alpha (\alpha +1)(\alpha +2)}{3!}+ x_{k-4}\frac{\alpha (\alpha +1)(\alpha +2)(\alpha +3)}{4!}+\cdots \right \} \\
 & = & x_{k}+x_{k-1}\{-\alpha +\alpha\}+x_{k-2}\left\{ \frac{\alpha (\alpha -1)}{2!} +\alpha^2+\frac{\alpha (\alpha +1)}{2!}  \right\} + \\
 &  & x_{k-3}\left\{-\frac{\alpha (\alpha -1)(\alpha -2)}{3!}+\frac{\alpha^2 (\alpha -1)}{2!}-\frac{\alpha^2 (\alpha +1)}{2!}+\frac{\alpha (\alpha +1)(\alpha +2)}{3!}  \right\}+ \\
 &  & x_{k-4}\left\{\frac{\alpha (\alpha -1)(\alpha -2)(\alpha -3)}{4!}-\frac{\alpha^2 (\alpha -1)(\alpha -2)}{3!}+\frac{\alpha^2 (\alpha +1)(\alpha -1)}{2!2!}-\frac{\alpha^2 (\alpha +1)(\alpha +2)}{3!}  \right\}+\cdots \\
 & = & x_{k}.
\end{array}
\end{eqnarray*}

\end{proof}

We refer to \cite{BalDut} for more properties of the fractional difference operators.

Note that the results of this work are more general and comprehensive then many other studies related to difference sequence spaces in literature.

\subsection{Preliminaries Results and Notations}

For the reader’s convenience, we state the known results that are used here and in the sequel.

Let $\omega $ denote the set of all complex sequences $x=(x_{k})_{k=0}^{\infty }$. We write $\ell _{\infty}$, $c$, $c _{0}$
and $\phi $ for the sets of all bounded, convergent, null and finite
sequences, respectively; also $bs $, $cs $ and $\ell _{1}$ denote
the sets of all bounded, convergent and absolutely convergent
series.

A subspace $X$ of $\omega $ is said to be a $BK$ space if it is a
Banach space with continuous coordinates $P_{n}:X\rightarrow \mathbb{C}$ 
$(n=0,1,\dots )$, where $P_{n}(x)=x_{n}$ for all $x\in X$. A $BK$ space $
X\supset \phi $ is said to have $AK$ if every sequence $x=(x_{k})_{k=0}^{
\infty }\in X$ has a unique representation $x=\lim_{m\rightarrow
\infty }x^{[m]}$, where $x^{[m]}=\sum_{n=0}^{m}x_{n}e^{(n)}$ is the
$m$ section of the sequence $x$. Let $X$ be a normed space. Then $
S_{X}=\{x\in X:\left\Vert x\right\Vert =1\}$ and $\bar{B}_{X}=\{x\in
X:\left\Vert x\right\Vert \leq 1\}$ denote the unit sphere and closed unit ball in $X$, where $X$ is a normed space. By $N_{r}$ we
denote any subset of $\mathbb{N}_{0}$ with elements greater or equal
to $r$.

Given any infinite matrix $A=(a_{nk})_{n,k=0}^{\infty }$ of complex
numbers and any sequence $x$, we write $A_{n}=(a_{nk})_{k=0}^{\infty
}$ for the
sequence in the $n^{th}$ row of $A$, $A_{n}x=\mbox{$\sum_{k=0}^{\infty}$}%
a_{nk}x_{k}$ $(n=0,1,\dots )$ and $Ax=(A_{n}x)_{n=0}^{\infty }$, provided $
A_{n}\in x^{\beta }$ for all $n$. If $X$ and $Y$ are subsets of
$\omega $, then
\begin{equation*}
X_{A}=\{x\in \omega :Ax\in X\}
\end{equation*}

denotes the matrix domain of $A$ in $X$
and $(X,Y)$ is the class of all infinite matrices that map $X$ into $Y$; so $%
A\in (X,Y)$ if and only if $X\subset Y_{A}$.

Given $a\in \omega $, we write
\begin{equation*}
\left\Vert a\right\Vert _{X}^{\ast }=\sup\limits_{x\in
S_{X}}\left\vert \sum\limits_{k=1}^{\infty }a_{k}x_{k}\right\vert
\end{equation*}%
provided the expression on the write hand side is defined and finite
which is the case whenever $X$ is a $BK$ space and $a\in X^{\beta }$ (\cite%
{Wil2}, Theorem 7.2.9, p. 107).

An infinite matrix $T=(t_{nk})_{n,k=0}^{\infty }$ is said to be a
triangle if $t_{nk}=0$ $(k>n)$ and $t_{nn}\not=0$ for all $n$.

A sequence $(b_{n})_{n=0}^{\infty }$ in a linear metric space $X$ is
called
a Schauder basis if for each $x\in X$ there exists a unique sequence $%
(\lambda _{n})_{n=0}^{\infty }$ of scalars such that
$x=\sum_{n=0}^{\infty }\lambda _{n}b_{n}$.

\subsection{The Difference Sequence Spaces of Fractional Order}

Consider now the following fractional difference sequence spaces.

\begin{eqnarray*}
c_{0}(\Delta ^{(  \tilde{\alpha} )} )&:=&
\left \{ x=(x_{k})\in\omega:
\lim_{k}\sum_{i=0}^{\infty}(-1)^{i}\frac{\Gamma(\tilde{\alpha} +1)}{i!\Gamma(\tilde{\alpha} -i+1)}x_{k-i}=0 \ \right \},\ \ \
\\
c(\Delta ^{(  \tilde{\alpha} )} )&:=&
\left \{ x=(x_{k})\in\omega:
\lim_{k}\sum_{i=0}^{\infty}(-1)^{i}\frac{\Gamma(\tilde{\alpha} +1)}{i!\Gamma(\tilde{\alpha} -i+1)}x_{k-i} \text{  exists} \right \}
, \
\\
\ell_{\infty} (\Delta ^{(  \widetilde{\alpha} )}  )&:=&\left \{x=(x_{k})\in\omega: \sup_{k}\left\vert\sum_{i=0}^{\infty}(-1)^{i}\frac{\Gamma(\tilde{\alpha} +1)}{i!\Gamma(\tilde{\alpha} -i+1)}x_{k-i}\right\vert<\infty\right \}.
\end{eqnarray*}

Let us now define the sequence $y =(y _{k})$ which will be
used, by the $\Delta ^{(  \tilde{\alpha} )}$-transform of a sequence $
x=(x_{k}) $, that is,

\begin{eqnarray*}
y_{k}&=& x_{k}-\tilde{\alpha} x_{k-1}+\frac{\tilde{\alpha} (\tilde{\alpha} -1)}{2!}x_{k-2}-\frac{\tilde{\alpha} (\tilde{\alpha} -1)(\tilde{\alpha} -2)}{3!}x_{k-3}+\cdots  \\
&=&\sum_{i=0}^{\infty}(-1)^{i}\frac{\Gamma(\tilde{\alpha} +1)}{i!\Gamma(\tilde{\alpha} -i+1)}x_{k-i}.
\end{eqnarray*}

Hence, those spaces can be considered as the matrix domains of the triangle $\Delta ^{(  \tilde{\alpha} )}$ in the classical sequence spaces $c_{0}, c, \ell_{\infty} $.
We also have the following relation between the sequences $x=(x_{k}) $ and $y=(y_{k}) $:

\begin{eqnarray*}
x_{k}=\sum_{i=0}^{\infty}(-1)^{i}\frac{\Gamma(-\tilde{\alpha} +1)}{i!\Gamma(-\tilde{\alpha} -i+1)}y_{k-i}.
\end{eqnarray*}

\begin{Lem}
\label{BKspace} (\cite{Wil2}, Theorem 4.3.12, p. 63) Let
$\left( X,\left\Vert .\right\Vert \right) $ be a BK space. Then
$X_{T}$\ is a BK space with $\left\Vert .\right\Vert _{T}=\left\Vert
T\left( .\right) \right\Vert .$
\end{Lem}
By Lemma \ref{BKspace}, defined fractional difference sequence spaces are complete, linear $BK$--spaces with the following norm:

\begin{eqnarray*}
\left\Vert x\right\Vert=\sup\limits_{n}\left\vert \sum_{i=0}^{\infty}(-1)^{i}\frac{\Gamma(\tilde{\alpha} +1)}{i!\Gamma(\tilde{\alpha} -i+1)}x_{n-i}\right\vert.
\end{eqnarray*}

\begin{Rem}
\label{rem1}(\cite{jarrah}, Remark 2.4)  The matrix domain $X_{T}$
of a linear metric sequence space $X$ has a basis if and only if $X$
has a basis.
\end{Rem}

\begin{Thm}
Let us define the sequences $c^{\left( n\right) }$ for $n=0,1,\dots $ and $%
c^{\left( -1\right) }$\ by

\begin{equation*}
c_{k}^{\left( n\right) }=\left\{
\begin{array}{lll}
0 &  & \left( 0\leq k<n\right) \\
(-1)^{k-n}\frac{\Gamma(-\tilde{\alpha} +1)}{(k-n)!\Gamma(-\tilde{\alpha} +n-k+1)}
&  & \left( k\geq n\right)%
\end{array}%
\right.
\end{equation*}

and

\begin{eqnarray*}
c_{k}^{\left( -1\right) }=\sum\limits_{n=0}^{k}(-1)^{k-n}\frac{\Gamma(-\tilde{\alpha} +1)}{(k-n)!\Gamma(-\tilde{\alpha} +n-k+1)}\ \text{  for  }\ k=0,1,\dots
\end{eqnarray*}

\begin{itemize}
\item[(i)] Then $\left( c^{\left( n\right) }\right) _{n=0}^{\infty }$ is a
\textit{Schauder }basis for $c_{0}(\Delta ^{(  \tilde{\alpha} )} )$ and
every sequence $x=\left( x_{n}\right) _{n=0}^{\infty }\in c_{0}(\Delta ^{(  \tilde{\alpha} )} )$ has a unique representation
\begin{eqnarray*}
x=\sum\limits_{n}(\Delta_{n} ^{(  \widetilde{\alpha} )}  x)c^{\left( n\right)  } \text{       }\forall n.
\end{eqnarray*}
\item[(ii)] Then $\left( c^{\left( n\right) }\right) _{n=-1}^{\infty }$ is a
\textit{Schauder }basis for $c(\Delta ^{(  \tilde{\alpha} )} ) $, and
every sequence $x=\left( x_{n}\right) _{n=0}^{\infty }\in c(\Delta ^{(  \tilde{\alpha} )} )$ has a unique representation

\begin{eqnarray*}
x=\xi c^{\left( -1\right) }+\sum\limits_{n}\left( y _{n}-\xi
\right) c^{\left( n\right) },
\end{eqnarray*}

where $\xi =\lim\limits_{n\rightarrow
\infty }y _{n}$.

\item[(iii)] The set $\ell_{\infty }(\Delta ^{(  \tilde{\alpha} )} )$ has no \textit{Schauder }basis.
\end{itemize}
\end{Thm}

\section{The $\beta$ Duals and Operator Norms of Fractional Spaces}
If $x$ and $y$ are sequences and $X$ and $Y$ are subsets of
$\omega$, then we write $x\cdot y=(x_{k}y_{k})_{k=0}^{\infty}$,

\begin{equation*}
x^{-1}*Y= \{a\in\omega:a\cdot x\in Y\}
\end{equation*}
and
\begin{equation*}
M(X,Y)=\bigcap_{x\in X}x^{-1}*Y= \{a\in \omega:a\cdot x\in Y\mbox{ for all }x\in X\}
\end{equation*}

for the multiplier space
of $X$ and $Y$; in particular, we use the notation $x^{\beta}=x^{-1}*cs$
$X^{\beta}=M(X,cs)$ for the $\beta$
dual of $X$.

\begin{Lem}\label{lem3.4}
Let $T$ be a triangle and $S$ be its inverse and $%
R=S^{t}$, the transpose of $S$.

\begin{itemize}
\item[(i)]Let $X$ be a BK space with AK or $X=\ell _{\infty }$. Then $a\in \left( X_{T}\right)
^{\beta }$ if and only if $a\in \left( X^{\beta }\right) _{R}$ and
$W\in \left(X,c_{0}\right) $ where the triangle $W$ is defined for $n=0,1,2,\ldots $ by%

\begin{equation}
w_{nk}=\left\{
\begin{array}{lll}
0 &  & \left( k>n\right) \\
\sum\limits_{j=n}^{\infty }a_{j}s_{jk} &  & \left( 0\leq k\leq
n\right).
\end{array}%
\right. \label{w}
\end{equation}%
Moreover, if $a\in \left( X_{T}\right) ^{\beta }$ then

\begin{equation}
\sum\limits_{k}a_{k}z_{k}=\sum\limits_{k}(R_{k}a)(T_{k}z) \text{\  } \forall z\in X_{T}. \label{lem11}
\end{equation}

\item[(ii)] We have $a\in \left( c_{T}\right) ^{\beta }$ if and only if $%
a\in \left( \ell _{1}\right) _{R}$ and $W\in \left( c,c\right) .$
Moreover, if $a\in \left( c_{T}\right) ^{\beta }$ then we have
\end{itemize}

\begin{eqnarray}
\sum\limits_{k} a_{k}z_{k}=\sum\limits_{k} (R_{k}a) (T_{k}z) -
\lim\limits_{k} T_{k}z \lim\limits_{n} \sum\limits_{k=0}^{n}w_{nk}
\text{\ } \forall z \in c_{T}.\label{lem12}
\end{eqnarray}
\end{Lem}

\begin{Rem}\label{rem3.5} We have the following results:
\begin{itemize}
\item[(i)] If $X$ be a BK space with AK then the condition $W\in \left(X,c_{0}\right) $ in Lemma \ref{lem3.4}(i) can be replaced by
\begin{eqnarray}
W\in \left(X,\ell_{\infty}\right).\label{RBeta}
\end{eqnarray}

\item[(ii)] The condition $W\in \left(c,c\right) $ in Lemma \ref{lem3.4}(ii) can be replaced by  the conditions
\begin{eqnarray}
W\in \left(c_{0},\ell_{\infty}\right) \text{\  and} \label{RBeta2}
\end{eqnarray}
\begin{eqnarray}
\lim\limits_{n}W_{n}e=\gamma \text{\  exists}.\label{RBeta3}
\end{eqnarray}
\end{itemize}

\end{Rem}

\begin{proof} By definition of the triangle $W$ we have
\begin{eqnarray}
\lim\limits_{n}W_{n}e^{(k)}=\lim\limits_{n}\sum\limits_{j=n}^{\infty
}=0 \text{\  } \forall k.\label{RBeta4}
\end{eqnarray}
\begin{itemize}
\item[(i)]The conditions (\ref{RBeta}) and (\ref{RBeta4}) imply $W\in \left(X,c_{0}\right) $ because $X$ has $AK$ and $c_{0}$ is a closed subspace of $\ell_{\infty}$. On the other hand clearly $W\in \left(X,c_{0}\right) $ implies (\ref{RBeta}).
\item[(ii)]Since $c_{0}$ has $AK$ and $c$ is a closed subspace of $\ell_{\infty}$, the conditions (\ref{RBeta2}) and (\ref{RBeta4}) imply $W\in \left(c_{0},c_{0}\right) $. Then $W\in \left(c_{0},c\right) $ and the condition (\ref{RBeta3}) imply $W\in \left(c,c\right) $. On the other hand $W\in \left(c,c\right) $ implies (\ref{RBeta2}) and $We\in c $, that is the condition (\ref{RBeta3}) holds.
\end{itemize}
\end{proof}

\begin{Thm}
\label{ourbeta}We have

\begin{itemize}
\item[(i)] $a\in \left( c_{0}(\Delta ^{(  \tilde{\alpha} )} )\right) ^{\beta }$ if
and only if
\begin{equation}
\sum\limits_{k}\left\vert \sum\limits_{j=k}^{\infty
}(-1)^{j-k}\frac{\Gamma(-\tilde{\alpha} +1)}{(j-k)!\Gamma(-\tilde{\alpha} -j+k+1)}a_{j}\right\vert <\infty\label{MF2}
\end{equation} and

\begin{equation}
\sup\limits_{n}\left( \sum\limits_{k=0}^{n}\left\vert
\sum\limits_{j=n}^{\infty }(-1)^{j-k}\frac{\Gamma(-\tilde{\alpha} +1)}{(j-k)!\Gamma(-\tilde{\alpha} -j+k+1)}
a_{j}\right\vert \right)
<\infty; \label{MF3}
\end{equation}
moreover, if $a\in \left( c_{0}(\Delta ^{(  \tilde{\alpha} )} )\right) ^{\beta }$ then $\forall x\in c_{0}(\Delta ^{(  \tilde{\alpha} )})$

\begin{equation}
\sum\limits_{k}a_{k}x_{k}=\sum\limits_{k}\left( \sum\limits_{j=k}^{\infty
}(-1)^{j-k}\frac{\Gamma(-\tilde{\alpha} +1)}{(j-k)!\Gamma(-\tilde{\alpha} -j+k+1)}a_{j}\right)y_{k}.\label{MMF3}
\end{equation}

\item[(ii)] $a\in \left( c(\Delta ^{(  \widetilde{\alpha} )} )\right) ^{\beta }$
if and only if (\ref{MF2}), (\ref{MF3}) and
\begin{equation}
\lim\limits_{n}\sum\limits_{k=0}^{n}\left( \sum\limits_{j=n}^{\infty }(-1)^{j-k}\frac{\Gamma(-\tilde{\alpha} +1)}{(j-k)!\Gamma(-\tilde{\alpha} -j+k+1)}
a_{j}\right) =\rho; \label{MF4}
\end{equation}

moreover, if $a\in \left( c(\Delta ^{(  \tilde{\alpha} )} )\right) ^{\beta }$ then $\forall x\in c(\Delta ^{(  \tilde{\alpha} )})$

\begin{equation*}
\sum\limits_{k}a_{k}x_{k}=\sum\limits_{k}\left( \sum\limits_{j=k}^{\infty
}(-1)^{j-k}\frac{\Gamma(-\tilde{\alpha} +1)}{(j-k)!\Gamma(-\tilde{\alpha} -j+k+1)}a_{j}\right)y_{k}-\rho\lim\limits_{k} y_{k}.
\end{equation*}

\item[(iii)] $a\in \left(\ell_{\infty} (\Delta ^{(  \widetilde{\alpha} )}  ) \right) ^{\beta }$ if
and only if (\ref{MF2}) and

\begin{equation}
\lim\limits_{n}\sum\limits_{k=0}^{n}\left\vert \sum\limits_{j=k}^{\infty
}(-1)^{j-k}\frac{\Gamma(-\tilde{\alpha} +1)}{(j-k)!\Gamma(-\tilde{\alpha} -j+k+1)}a_{j}\right\vert =0; \label{MF5}
\end{equation}

moreover, if $a\in \left( \ell_{\infty}(\Delta ^{(  \tilde{\alpha} )} )\right) ^{\beta }$ then (\ref{MMF3}) holds $\forall x\in \ell_{\infty} (\Delta ^{(  \tilde{\alpha} )})$.
\end{itemize}
\end{Thm}

\begin{proof}
We apply Lemma \ref{lem3.4} and Remark \ref{rem3.5}.

The triangles $R$\ and $W$ defined for $n=0,1,\ldots $ by%
\begin{equation}
R_{k}a=\sum\limits_{j=k}^{\infty }s_{jk}a_{j}=\sum\limits_{j=k}^{\infty
}(-1)^{j-k}\frac{\Gamma(-\tilde{\alpha} +1)}{(j-k)!\Gamma(-\tilde{\alpha} -j+k+1)}a_{j}\text{ \ }(k=0,1,\ldots ) \label{Rka}
\end{equation}%
and%
\begin{equation}
w_{nk}=\left\{
\begin{array}{ccl}
\sum\limits_{j=n}^{\infty
}(-1)^{j-k}\frac{\Gamma(-\tilde{\alpha} +1)}{(j-k)!\Gamma(-\tilde{\alpha} -j+k+1)}a_{j} &  & (0\leq k\leq n) \\
0 &  & (k>n).%
\end{array}\label{pwnk}
\right.
\end{equation}

The condition $Ra\in \ell _{1}$ of Lemma \ref{lem3.4} holds in each part because $c_{0}^{\beta}=c^{\beta}=\ell_{\infty}^{\beta}$.

By Lemma \ref{lem3.4} and Remark \ref{rem3.5}(ii), we must add the condition $W\in (c_{0},\ell _{\infty})$ which is equivalent to
\begin{equation*}
\sup\limits_{n} \sum_{k=0}^{n}\left\vert w_{nk}\right\vert < \infty
\end{equation*}

and this is the condition in (\ref{RBeta3}) to (\ref{MF3}), and (\ref{MF3}) is the condition in (\ref{MF4}).

By Lemma \ref{lem3.4}, we must add the condition $W\in (\ell _{\infty},c_{0})$ which is equivalent to
\begin{equation*}
\lim\limits_{n}\sum\limits_{k=0}^{n}\left%
\vert w_{nk}\right\vert =0.
\end{equation*}

Note that, (\ref{Rka}) in Parts (i) and (iii) and (\ref{pwnk}) come from (\ref{lem11}) and (\ref{lem12}), respectively.

\end{proof}
\subsection{Operators Norms  and Matrix Transformations of Fractional Sequence Spaces}
Let us now establish identities and inequalities of operator norms for fractional sequence spaces and then characterize some classes of matrix mappings on them. The following notations and results are needed to characterize certain classes of matrix mappings on the sets of fractional sequences and for determination of the operator norms of our spaces.

\begin{Lem}
Let $X$ and $Y$ be $BK$ spaces
\begin{itemize}
\item[(i) ] Then we have $\left( X,Y\right)\subset\mathcal{B}\left( X,Y\right)$, that is,
every  $A\in\left( X,Y\right)$ defines an operator
$L_{A}\in\mathcal{B}\left( X,Y\right)$, where $L_{A}(x)=Ax$ for all
$x\in X$ (see \cite{mara}, Theorem 1.23).
\item[(ii) ] If $X$ has AK then we have $\mathcal{B}\left( X,Y\right)\subset\left( X,Y\right)$, that is,
every  operator  $L\in\mathcal{B}\left( X,Y\right)$ is given by a
matrix $A\in\left( X,Y\right)$ such that $L(x)=Ax$ for all $x\in X$
(see \cite{jarrah}, Theorem 1.9).
\end{itemize}
\end{Lem}

\begin{Lem}\label{operatornorm.1}
Let $Y$ be an arbitrary subset of $\omega$.
\begin{itemize}
\item[(i) ] Let $X$ be a $BK$ space with AK or $X=\ell_{\infty}$,
and $R=S^{t} $. Then $A\in(X_{T},Y)$
if and only if $\hat{A}\in\left( X,Y\right)$ and
$W^{(A_{n})}\in(X,c_{0})$ for all $n=0,1,\ldots $. Here $\hat{A}$ is
the matrix with rows $\hat{A}_{n}=RA_{n}$ for $n=0,1,\dots$, and the
triangles $W^{(A_{n})}$ $(n=0,1,\ldots)$ are defined as in
\ref{matriceW} with $a_{j}$ replaced by $a_{nj}$.

Moreover, if $A\in\left( X,Y\right)$ then we have $Az=\hat{A}(Tz)$
for all $z\in Z=X_{T}$ (see \cite{Mara7}, Theorem 3.4,
Remark 3.5(a)).

\item[(ii) ] We have $A\in(c_{T},Y)$ if and only if
$\hat{A}\in \left( c_{0},Y\right) $
and $W^{(A_{n})}\in \left( c,c\right) $ for all $n=0,1,\ldots $ and $\hat{A}%
e-(\gamma _{n})_{n=0}^{\infty }\in Y$, where

\begin{equation*}
\gamma
_{n}=\lim\limits_{m}\sum\limits_{k=0}^{m}w_{mk}^{(A_{n})} \text{for } n=0,1,\ldots
\end{equation*}

Moreover, if $A\in \left( c_{T},Y\right) $ then we have $Az=\hat{A}
(Tz)-\eta (\gamma _{n})_{n=0}^{\infty }$ for all $z\in c_{T}$, where
$\eta =\lim\limits_{k}T_{k}z$ (\cite{Mara7}, Theorem
1.23).
\end{itemize}
\end{Lem}

\begin{Thm}\label{operatornorm.2} Let $X=c_{0}(\Delta ^{(  \tilde{\alpha} )} )$ or $X=\ell_{\infty} (\Delta ^{(  \widetilde{\alpha} )}  )$.

\begin{itemize}

\item[(i) ] Let $Y=c_{0},c,\ell_{\infty}$. If  $A\in\left( X_{T},Y\right)$
then, putting
\begin{equation*}
\left\Vert A\right\Vert _{(X_{T},\infty )}=\sup_{n}\left\Vert \hat{A}%
_{n}\right\Vert_{1} =\sup_{n} \sum\limits_{k}\left\vert \sum\limits_{j=k}^{\infty
}(-1)^{j-k}\frac{\Gamma(-\tilde{\alpha} +1)}{(j-k)!\Gamma(-\tilde{\alpha} -j+k+1)}a_{nj}\right\vert ,
\end{equation*}
we have $\left\Vert L_{A}\right\Vert=\left\Vert A\right\Vert
_{(X_{T},\infty )}$.

\item[(ii) ] Let $Y=\ell_{1}$. If  $A\in\left( X_{T},\ell_{1}\right)$. Then we have

\begin{equation*}
\left\Vert A\right\Vert _{\left( X_{T},1\right) }=\sup_{\substack{ N\subset \mathbb{N}
\\ N\text{ finite}}}\left( \sum\limits_{k}\left\vert
\sum\limits_{n\in N}\sum\limits_{j=k}^{\infty
}(-1)^{j-k}\frac{\Gamma(-\tilde{\alpha} +1)}{(j-k)!\Gamma(-\tilde{\alpha} -j+k+1)}a_{nj}\right\vert \right)\leq\left\Vert L_{A}\right\Vert\leq
4\left\Vert A\right\Vert _{\left( X_{T},1\right) }.
\end{equation*}
\end{itemize}
\end{Thm}

\begin{proof}
The proof is based on the results in (\cite{djolovıc}, Theorem 2.8)
\end{proof}

\begin{Thm}\label{operatornorm.3} The operator norm of the set  $c(\Delta ^{(  \widetilde{\alpha} )} )$ is given.
\begin{itemize}
\item[(i) ]
Let $A\in\left(c(\Delta ^{(  \widetilde{\alpha} )} ),Y\right)$, where $Y$is any of the spaces
$c_{0},c$ or $\ell_{\infty}$. Then we have
\begin{equation*}
\left\Vert L_{A}\right\Vert =\left\Vert A\right\Vert _{\left(
c(\Delta ^{(  \widetilde{\alpha} )} ),\infty
\right) }=\sup_{n}\left( \sum\limits_{k}\left\vert \sum\limits_{j=k}^{\infty
}(-1)^{j-k}\frac{\Gamma(-\tilde{\alpha} +1)}{(j-k)!\Gamma(-\tilde{\alpha} -j+k+1)}a_{nj}\right\vert +\left\vert \gamma _{n}\right\vert \right),
\end{equation*}

where $ \gamma
_{n}=\lim_{m}\sum_{k=0}^{m}w_{mk}^{(A_{n})} \text{for } n=0,1,\ldots $

\item[(ii) ]
Let $A\in\left( c(\Delta ^{(  \widetilde{\alpha} )} ),\ell_{1}\right)$, then, putting,

\begin{equation*}
\left\Vert A\right\Vert _{\left( c(\Delta ^{(  \widetilde{\alpha} )} ),1\right) }=\sup_{\substack{ N\subset \mathbb{N}
\\ N\text{ finite}}}\left( \sum\limits_{k}\left\vert
\sum\limits_{n\in N}\sum\limits_{j=k}^{\infty
}(-1)^{j-k}\frac{\Gamma(-\tilde{\alpha} +1)}{(j-k)!\Gamma(-\tilde{\alpha} -j+k+1)}a_{nj}\right\vert +\left\vert \sum\limits_{n\in N}\gamma
_{n}\right\vert \right)
\end{equation*}

we have

\begin{equation*}
\left\Vert A\right\Vert _{\left( c(\Delta ^{(  \widetilde{\alpha} )} ),1\right) } \leq \left\Vert L_{A}\right\Vert \leq
4\left\Vert A\right\Vert _{\left( c(\Delta ^{(  \widetilde{\alpha} )} ),1\right) }.
\end{equation*}
\end{itemize}

\end{Thm}

\begin{proof}
The proof is based on the results in (\cite{djolovıc}, Theorem 2.9)
\end{proof}

\begin{Thm}
\label{EMAPAPER.S.4.1}
The necessary and sufficient conditions for $A\in(\ell_{\infty} (\Delta ^{(  \widetilde{\alpha} )}  ),Y)$ $A\in(c_{0}(\Delta ^{(  \tilde{\alpha} )} ),Y)$ and $A\in(c(\Delta ^{(  \widetilde{\alpha} )} ),Y)$, where $Y\in \left\lbrace \ell_{\infty},c_{0},c,\ell_{1}\right\rbrace $ can be read from the following
table:
\begin{center}
\begin{tabular}{||l|c|c|c||}\hline
    \begin{tabular}[t]{lr}
        & From\\
        To &
    \end{tabular}
    &  $\ell_{\infty} (\Delta ^{(  \widetilde{\alpha} )}  )$ & $c_{0}(\Delta ^{(  \tilde{\alpha} )} )$ & $c(\Delta ^{(  \tilde{\alpha} )} )$\\
    $\ell_{\infty}$ & {\bf 1.} & {\bf 1.} & {\bf 2.}\\
    $c_{0}$ & {\bf 3.} & {\bf 3.} & {\bf 4.}\\
    $c$ & {\bf 5.} & {\bf 5.} & {\bf 6.}\\
    $\ell_{1}$ & {\bf 7.} & {\bf 7.} & {\bf 8.}\\\hline
\end{tabular}
\end{center}
\end{Thm}
\begin{itemize}
	\item[[\textbf{1}] ] 1A and 1B where
	\begin{itemize}
		\item[(1A)] $\Vert A\Vert _{(c(\Delta ^{(  \widetilde{\alpha} )} ),\infty)}=\sup\limits_{n}\sum\limits_{k=0}^{\infty }\left\vert \hat{a}_{nk}\right\vert <\infty $,
		\item[(1B)]  $\Vert W^{(A_{n})}\Vert _{(\ell _{\infty},c_{0})}=\lim\limits_{m\rightarrow \infty}\sum\limits_{k=0}^{m}\left\vert w_{mk}^{(A_{n})}\right\vert =0$ for all $n$.
	\end{itemize}
	\item[[\textbf{2}]] 1A and 2A where
	\begin{itemize}
		\item[(2A)] $\Vert W^{(A_{n})}\Vert _{(\ell _{\infty },\ell _{\infty})}=\sup\limits_{m}\sum\limits_{k=0}^{m}\left\vert w_{mk}^{(A_{n})}\right\vert <\infty $ for all $n$.
	\end{itemize}
	\item[[\textbf{3}]] 1A, 2A, 3A and 3B where
	\begin{itemize}
		\item[(3A)] $\lim\limits_{m\rightarrow \infty}\sum\limits_{k=0}^{m}w_{mk}^{(n)}=\gamma _{n}$ exists for each $n$,
		\item[(3B)] $\sup\limits_{n}\left\vert \sum\limits_{k=0}^{\infty }\hat{a}_{nk}-\gamma _{n}\right\vert =0$.
	\end{itemize}
	\item[[\textbf{4}]] 1B and 4A where
	\begin{itemize}
		\item[(4A)] $ \lim\limits_{m\rightarrow \infty}\sum\limits_{k=0}^{m}\left\vert \hat{a}_{nk}\right\vert =0$.
	\end{itemize}
	\item[[\textbf{5}]] 1A, 2A and 5A where
	\begin{itemize}
		\item[(5A)] $\lim\limits_{n\rightarrow \infty }\hat{a}_{nk}=0$ for each $k$.
	\end{itemize}
	\item[[\textbf{6}]] 1A, 2A, 3A, 5A and 6A where
	\begin{itemize}
		\item[(6A)] $\lim\limits_{n\rightarrow \infty }\left( \sum\limits_{k=0}^{\infty }\hat{a}_{nk}-\gamma _{n}\right) =0$.
	\end{itemize}
	\item[[\textbf{7}] ] 1B, 7A, 7B and 7C where
	\begin{itemize}
		\item[(7A)] $\lim\limits_{n\rightarrow \infty }\hat{a}_{nk}=\hat{\alpha}_{k}$ exists for each $n$,
		\item[(7B)] $\sum\limits_{k=0}^{\infty }\left\vert \hat{a}_{nk}\right\vert ,\sum\limits_{k=0}^{\infty }\left\vert \hat{\alpha}_{k}\right\vert <\infty $ for all $n$,
		\item[(7C)] $\lim\limits_{n\rightarrow \infty }\left( \sum\limits_{k=0}^{\infty }\hat{a}_{nk}-\hat{\alpha}_{k}\right) =0$.
	\end{itemize}
	\item[[\textbf{8}] ] 1A, 2B and 7A.
	\item[[\textbf{9}]] 1A, 2B, 3A, 7A and 9A where
	\begin{itemize}
		\item[(9A)] $\lim\limits_{n\rightarrow \infty }\left( \sum\limits_{k=0}^{\infty }\hat{a}_{nk}-\gamma _{n}\right) =\delta$ exists. 
	\end{itemize}
	\item[[\textbf{10}]] 1B and 10A where
	\begin{itemize}
		\item[(10A)] $\sup\limits_{K\subset \mathbb{N}}\sum\limits_{n=0}^{\infty}\left\vert \sum\limits_{k\in K}\hat{a}_{nk}\right\vert <\infty $.
		\end{itemize}
		\item[[\textbf{11}] ] 2A and 10A.
		\item[[\textbf{12}]] 2A, 3A, 10A and 12A where
		\begin{itemize}
		\item[(12A)] $\sum\limits_{n=0}^{\infty }\left\vert \sum\limits_{k=0}^{\infty }\hat{a}_{nk}-\gamma _{n}\right\vert <\infty $, 
		\end{itemize}
	\end{itemize}
where for given a matrix $A=(a_{nk})_{n,k=0}^{\infty }$, we define the matrices $\hat{A}=(\hat{a}_{nk})_{n,k=0}^{\infty }$\ and $W^{(A_{n})}=(w_{mk}^{(A_{n})})_{m,k=0}^{\infty }$ by

\begin{equation}
\hat{a}_{nk}=\sum\limits_{j=k}^{\infty
}(-1)^{j-k}\frac{\Gamma(-\tilde{\alpha} +1)}{(j-k)!\Gamma(-\tilde{\alpha} -j+k+1)}a_{nj}\text{
	for all }n,k\in \mathbb{N}_{0}  \label{matriceA}
\end{equation}%
and
\begin{equation}
w_{mk}^{(A_{n})}=\left\{\begin{array}{cc}
\sum\limits_{j=m}^{\infty
}(-1)^{j-k}\frac{\Gamma(-\tilde{\alpha} +1)}{(j-k)!\Gamma(-\tilde{\alpha} +j-k+1)}a_{nj} & (0\leq k\leq m) \\
0 & (k>m)
\end{array}
\right. \label{matriceW}
\end{equation}

for $n,m\in \mathbb{N}_{0}$,

\begin{proof}
Note that the entries of the triangles $\hat{A}$ and
$W^{(A_{n})}$ are given above and assume that $Y\in \left\lbrace c_{0},c,\ell_{\infty},\ell_{1}\right\rbrace $. \\

\underline{[1], [4], [7], [10]} Taking into account Lemma \ref{operatornorm.1}(i) we have $A\in (\ell_{\infty }(\Delta ^{(  \tilde{\alpha} )} ),Y)$
if and only if $\hat{A}\in \left( \ell _{\infty },Y\right)$ and $W^{(A_{n})}\in \left( \ell _{\infty },c_{0}\right)$ for each $n$.
First,  $\hat{A}\in \left( \ell _{\infty },Y\right)$ satisfies (1A) in [1], by [\cite{Wil2}, Theorem 1.3.3], (4A) in [4] by [\cite{Stieglitz},  21. (21.1)], (7A) , (7B) and (7C) in [7] by [\cite{Wil2}, Theorem 1.7.18 (ii)], and (10A) in [10] by [\cite{Wil2}, 8.4.9A]. Also, $W^{(A_{n})}\in \left( \ell _{\infty },c_{0}\right)$ for all $n$ satisfies (2A) in [1], [4], [7], [10] by [\cite{Wil2}, Theorem 1.3.3].

\underline{[2], [5], [8], [11]} Remark \ref{rem3.5}(i) and Lemma \ref{operatornorm.1}(i) satisfy $A\in (\ell_{\infty }(\Delta ^{(  \tilde{\alpha} )} ),Y)$\ if and only if $\hat{A}\in \left(c_{0},Y\right) $ and $W^{(A_{n})}\in \left(c_{0},c_{0}\right) $ for each $n=0,1,\ldots $. First $\hat{A}\in \left( c_{0}, c_{0}\right)$  satisfies (1A) in [2] by [\cite{Wil2}, Theorem 1.3.3], (1A) and (5A) in [5] by [\cite{Wil2}, 8.4.5A], (1A) and (7A) in [8] by [\cite{Wil2}, 8.4.5A] and (10A)
in [11] by [15, 8.4.3B]. Also by [\cite{Wil2}, Theorem 1.3.3] $W^{(A_{n})}\in\left(c_{0},c_{0}\right)$ for all $  n $ satisfies (2A) in [2], [5], [8], [11].

\underline{[3], [6], [9], [12]} Remark \ref{rem3.5}(i) and Lemma \ref{operatornorm.1}(i) satisfy $A\in (c_{0}(\Delta ^{(  \tilde{\alpha} )} ),Y)$ if and only if $\hat{A}\in \left( c_{0}, Y\right)$ and $W^{(A_{n})}\in\left(c_{0},\ell_{\infty}\right)$ for each $n$, and 

\begin{equation*}
\lim\limits_{n\rightarrow \infty }W_{m}^{(A_{n})}=\lim_{m}\sum_{k=0}^{m}w_{mk}^{(A_{n})} \text{  for } n=0,1,\ldots 
\end{equation*}

and for each $ n $

\begin{equation*}
\left\lbrace \hat{A}-\left( \lim_{m}\sum_{k=0}^{m}w_{mk}^{(A_{n})}\right) \right\rbrace \in Y.
\end{equation*}

It means, we have to add the last two conditions to those for $A \in(c_{0 }(\Delta ^{(  \tilde{\alpha} )} ),Y)$, that is, (3A) and (3B) in [3] to those in [2], (3A) and (6A) in 6. to those in 5., (3A) and (9A) in [9] to those in [8] and (3A) and (12A) in [12] to those in [11].
\end{proof}

\section{Compact Operators on Fractional Spaces $c_{0}(\Delta ^{(  \tilde{\alpha} )} ), c(\Delta ^{(  \widetilde{\alpha} )} )$ and $\ell_{\infty} (\Delta ^{(  \widetilde{\alpha} )}  )$}
In this section, we give our main results related to compact operators on fractional sequence spaces. We recall the definition of the Hausdorff measure of noncompactness
of bounded subsets of a metric space, and the Hausdorff measure of
noncompactness of operators between Banach spaces.

If $X$ and $Y$ are infinite--dimensional complex Banach spaces then
a linear operator $L:X\rightarrow Y$ is said to be compact if the
domain of $L$ is
all of $X$, and, for every bounded sequence $(x_{n})$ in $X$, the sequence $%
(L(x_{n}))$ has a convergent subsequence. We denote the class of
such operators by $\mathcal{C}(X,Y)$.

Let $(X,d)$ be a metric space, $B(x_{0},\delta )=\{x\in
X:d(x,x_{0})<\delta \}$ denote the open ball of radius $\delta >0$
and center in $x_{0}\in X$, and $\mathcal{M}_{X}$ be the collection
of bounded sets in $X$. The Hausdorff measure of noncompactness of
$Q\in \mathcal{M}_{X}$ is
\begin{equation*}
\chi (Q)=\inf \{\epsilon > 0:Q\subset
\bigcup_{k=1}^{n}B(x_{k},\delta _{k}):x_{k}\in X,\ \delta
_{k}<\epsilon,\ 1\leq k \leq n,\ n\in\mathbb{N}\}.
\end{equation*}%
Let $X$ and $Y$ be Banach spaces and $\chi _{1}$ and $\chi _{2}$ be
measures of noncompactness on $X$ and $Y$. Then the operator
$L:X\rightarrow Y$ is called $(\chi _{1},\chi _{2})$--bounded if
$L(Q)\in \mathcal{M}_{Y}$ for every $Q\in \mathcal{M}_{X}$ and there
exists a positive constant $C$ such that
\begin{equation}
\chi _{2}(L(Q))\leq C\chi _{1}(Q)\mbox{ for every }Q\in
\mathcal{M}_{X}. \label{EmaFarAb.S.3.Eq.1}
\end{equation}%
If an operator $L$ is $(\chi _{1},\chi _{2})$--bounded then the
number
\begin{equation*}
\Vert L\Vert _{(\chi _{1},\chi _{2})}=\inf \left\{ C\geq 0:(\ref%
{EmaFarAb.S.3.Eq.1})\mbox{ holds for all }Q\in
\mathcal{M}_{X}\right\}
\end{equation*}%
is called the $(\chi _{1},\chi _{2})$--measure of noncompactness of
$L$. In particular, if $\chi _{1}=\chi _{2}=\chi $, then we write
$\Vert L\Vert _{\chi }$ instead of $\Vert L\Vert _{(\chi ,\chi )}$.

\begin{Lem} (\cite{mara}, Theorem 2.25 and Corollary 2.26)
\label{compactoperator}Let $X$ and $Y$ are Banach spaces and $L\in \mathcal{B%
}\left( X,Y\right) $. Then we have
\begin{equation}
\left\Vert L\right\Vert _{\chi }=\chi \left( L(\bar{B}_{X})\right)
=\chi \left( L(S_{X})\right) ,  \label{c13}
\end{equation}%
\begin{equation}
L\in \mathcal{C}(X,Y)\text{ if and only if }\left\Vert L\right\Vert
_{\chi }=0.  \label{c14}
\end{equation}%
\end{Lem}

\begin{Lem}
(Goldenstein, Gohberg, Markus \cite{mara}, Theorem 2.23) Let $X$ be
a Banach space with Schauder basis $\left( b_{n}\right) _{n=0}^{\infty }$, $%
Q\in \mathcal{M}_{X},$ $P_{n}:X\rightarrow X$ be the projector onto
the linear span of $\left\{ b_{0},b_{1},\ldots b_{n}\right\} $. $I$
be the identity map on $X$ and $R_{n}=I-P_{n}$ $(n=0,1,\dots)$.
Then we have%
\begin{equation}
\dfrac{1}{a}\cdot \limsup_{n\rightarrow \infty } \left( \sup_{x\in
Q}\left\Vert  R_{n} (x)\right\Vert \right) \leq \chi (Q)\leq
\limsup_{n\rightarrow \infty } \left( \sup_{x\in Q}\left\Vert R_{n}
(x)\right\Vert \right), \label{c16}
\end{equation}%
where $a=\limsup_{n\rightarrow \infty } \left\Vert
R_{n}\right\Vert$.
\end{Lem}

\begin{Lem}
(\cite{mara}, Theorem 2.8) Let $Q$\ be a bounded subset of the
normed space $X$, where $X$ is $\ell _{p}$\ for $1\leq p<\infty $ or $c_{0}$. If\ $%
P_{n}:X\rightarrow X$ is the operator defined by $P_{n}(x)=x^{[n]}$ for $%
x=(x_{k})_{k=0}^{\infty }\in X$, then we have

\begin{equation*}
\chi (Q)=\lim_{n}\left( \sup_{x\in Q}\left\Vert R_{n} (x)\right\Vert
\right).
\end{equation*}
\end{Lem}

The final results of this section give the estimates of the
Hausdorff measure of noncompactness of $L_{A}$ when $A\in (X_{T},c)$
for $X=c_{0},\ell _{\infty },c$.

\begin{Lem}
( \cite{djolovic2}, Corollary 5.13) If $A\in \left(
(c_{0})_{T},c\right) $ or $A\in \left( (\ell _{\infty
})_{T},c\right) $\ then we have
\begin{equation}
\dfrac{1}{2}\cdot \lim\limits_{r\rightarrow \infty }\left(
\sup\limits_{n\geq r}\left\Vert \hat{A}_{n}-\hat{\alpha}\right\Vert
_{1}\right) \leq \Vert L_{A}\Vert _{\chi }\leq
\lim\limits_{r\rightarrow
\infty }\left( \sup\limits_{n\geq r}\left\Vert \hat{A}_{n}-\hat{\alpha}%
\right\Vert _{1}\right),  \label{c17}
\end{equation}%
where $\hat{\alpha}=\left( \alpha _{k}\right) _{k=0}^{\infty }$ with $\hat{%
\alpha}_{k}=\lim_{n\rightarrow \infty }\hat{a}_{nk}$ for
$k=0,1,\dots $
\end{Lem}

\begin{Lem}
( \cite{djolovic2}, Corollary 5.14) If $A\in \left( c_{T},c\right)
$\
then we have%
\begin{eqnarray}
&&\dfrac{1}{2}\cdot \lim\limits_{r\rightarrow \infty }\left(
\sup\limits_{n\geq r}\left( \left\vert \beta -\delta
_{n}-\sum\limits_{k=0}^{\infty }\alpha _{k}\right\vert
+\sum\limits_{k=0}^{\infty }\left\vert \hat{a}_{nk}-\alpha
_{k}\right\vert
\right) \right)  \label{c18} \\
&\leq &\Vert L_{A}\Vert _{\chi }\leq \lim\limits_{r\rightarrow
\infty }\left( \sup\limits_{n\geq r}\left( \left\vert \beta -\delta
_{n}-\sum\limits_{k=0}^{\infty }\alpha _{k}\right\vert
+\sum\limits_{k=0}^{\infty }\left\vert \hat{a}_{nk}-\alpha
_{k}\right\vert \right) \right),  \notag
\end{eqnarray}%
where $\gamma _{n}=\lim_{m\rightarrow \infty }w_{mk}^{(A_{n})}$ for $%
n=0,1,\dots $, $\beta =\lim_{n\rightarrow \infty }(\sum_{k=0}^{\infty }\hat{a%
}_{nk}-\gamma _{n})$ and $\hat{\alpha}=\left( \alpha _{k}\right)
_{k=0}^{\infty }$ with $\hat{\alpha}_{k}=\lim_{n\rightarrow \infty }\hat{a}%
_{nk}$ for $k=0,1,\dots $.
\end{Lem}

We now establish necessary and sufficient conditions
for a matrix operator to be a compact operator from fractional difference sequence spaces into $Y$, where $Y \in \left\lbrace c_{0},c,\ell_{\infty},\ell_{1}\right\rbrace $.
This is achieved applying the results given above about
Hausdorff measure of noncompactness.

\begin{Thm}\label{generaltheorem}
The identities or estimates
for $L_{A}$ when $A\in(\ell_{\infty} (\Delta ^{(  \widetilde{\alpha} )}  ),Y)$, $A\in(c_{0}(\Delta ^{(  \tilde{\alpha} )} ),Y)$ and $A\in(c(\Delta ^{(  \widetilde{\alpha} )} ),Y)$, where $Y\in \left\lbrace \ell_{\infty},c_{0},c,\ell_{1}\right\rbrace $ can be read from the following
table:

\begin{center}
\begin{tabular}{||l|c|c|c||}\hline
    \begin{tabular}[t]{lr}
        & From\\
        To &
    \end{tabular}
    &  $\ell_{\infty} (\Delta ^{(  \widetilde{\alpha} )}  )$ & $c_{0}(\Delta ^{(  \tilde{\alpha} )} )$ & $c(\Delta ^{(  \tilde{\alpha} )} )$\\
    $\ell_{\infty}$ & {\bf 1.} & {\bf 1.} & {\bf 2.}\\
    $c_{0}$ & {\bf 3.} & {\bf 3.} & {\bf 4.}\\
    $c$ & {\bf 5.} & {\bf 5.} & {\bf 6.}\\
    $\ell_{1}$ & {\bf 7.} & {\bf 7.} & {\bf 8.} \\\hline
\end{tabular}
\end{center}
Here
\par
\begin{tabular}{l@{\quad}l}
    {\bf 1.} & 
    $0\leq \Vert L_{A}\Vert _{\chi }\leq \lim\limits_{r\rightarrow \infty}\left( \sup\limits_{n\geq r}\sum\limits_{k=0}^{\infty }\left\vert \hat{a}
    _{nk}\right\vert \right) $;\\
    {\bf 2.} &
    $0\leq \Vert L_{A}\Vert _{\chi }\leq \lim\limits_{r\rightarrow \infty
    }\left( \sup\limits_{n\geq r}\sum\limits_{k=0}^{\infty }\left\vert \hat{a}
    _{nk}\right\vert +\left\vert \gamma _{n}\right\vert \right) $;\\
    {\bf 3.} &
    $\Vert L_{A}\Vert _{\chi }=\lim\limits_{r\rightarrow \infty}\left\Vert \hat{A}^{[r]}\right\Vert _{(\ell _{\infty },\ell _{\infty })}$;\\
    {\bf 4.} & 
    $\Vert L_{A}\Vert _{\chi }=\lim\limits_{r\rightarrow \infty }\left(\sup\limits_{n\geq r}\sum\limits_{k=0}^{\infty }\left\vert \hat{a}_{nk}\right\vert +\left\vert \gamma _{n}\right\vert \right) $;\\
    {\bf 5.} &
    $\dfrac{1}{2}\cdot \lim\limits_{r\rightarrow \infty }\left\Vert \hat{B}^{[r]}\right\Vert _{((\ell _{\infty },\ell _{\infty })}\leq \Vert L_{A}\Vert
    _{\chi }\leq \lim\limits_{r\rightarrow \infty }\left\Vert \hat{B}^{[r]}\right\Vert _{((\ell _{\infty },\ell _{\infty })}$;\\
    {\bf 6.} &
    $\dfrac{1}{2}\cdot \lim\limits_{r\rightarrow \infty }\left(
    \sup\limits_{n\geq r}\sum\limits_{k=0}^{\infty }\left\vert \hat{b}
    _{nk}\right\vert +\left\vert \delta _{n}\right\vert \right) \leq
    \Vert L_{A}\Vert _{\chi }\leq \lim\limits_{r\rightarrow \infty
    }\left(\sup\limits_{n\geq r}\sum\limits_{k=0}^{\infty }\left\vert \hat{b}
    _{nk}\right\vert +\left\vert \delta _{n}\right\vert \right) $;\\
    {\bf 7.} & 
    $\lim\limits_{r\rightarrow \infty }\sup\limits_{\substack{ N\subset\mathbb{N} _{0} \\ \text{finite}}}\left\Vert \sum\limits_{n\in
	\mathbb{N} }\hat{A}_{n}^{[r]}\right\Vert _{1}\leq \Vert L_{A}\Vert
	_{\chi }\leq 4\lim\limits_{r\rightarrow \infty
	}\sup\limits_{\substack{ N\subset \mathbb{N} _{0} \\
	\text{finite}}}\left\Vert \sum\limits_{n\in \mathbb{N}
	}\hat{A}_{n}^{[r]}\right\Vert _{1}$
	;\\
    {\bf 8.} & 
    $\lim\limits_{r\rightarrow \infty }\sup\limits_{\substack{ N\subset
	\mathbb{N}_{0} \\ \text{finite}}}\left( \left\Vert \sum\limits_{n\in \mathbb{
	N}}\hat{A}_{n}^{[r]}\right\Vert _{1}+\left\vert \sum\limits_{n\in
	N}\gamma _{n}\right\vert \right) $

	$\leq \Vert L_{A}\Vert _{\chi }\leq 4\lim\limits_{r\rightarrow
	\infty }\sup\limits_{\substack{ N\subset \mathbb{N}_{0} \\
	\text{finite}}}\left( \left\Vert \sum\limits_{n\in
	\mathbb{N}}\hat{A}_{n}^{[r]}\right\Vert _{1}+\left\vert
	\sum\limits_{n\in N}\gamma _{n}\right\vert \right) $,
\end{tabular}
where the notations used in the theorem are defined as follows:

Let $A=(a_{nk})_{n,k=0}^{\infty }$ be an infinite matrix and $r\in \mathbb{N}
_{0}$. Then $A^{[r]}$ denotes the matrix with rows $A_{n}^{[r]}=0$\ for $%
0\leq n\leq r$ and $A_{n}^{[r]}=A_{n}$\ for $n\geq r+1$.

We write $\hat{A}$ for the matrix with
\begin{equation*}
\hat{a}_{nk}=\sum\limits_{j=k}^{\infty
}(-1)^{j-k}\frac{\Gamma(-\tilde{\alpha} +1)}{(j-k)!\Gamma(-\tilde{\alpha} -j+k+1)}a_{nj}\text{
	for all }n,k\in \mathbb{N}_{0}\text{;}
\end{equation*}

and $\hat{\alpha}=(\hat{\alpha}_{k})_{k=0}^{\infty }$ and $\gamma
=(\gamma _{n})_{n=0}^{\infty }$\ for the sequences with
$\hat{\alpha}_{k}=\lim_{n\rightarrow \infty }\hat{a}_{nk} $ for
$k=0,1,\ldots $ and
\begin{equation*}
\gamma _{n}=\lim\limits_{m}\sum\limits_{k=0}^{m}w_{mk}^{(A_{n})}=\lim\limits_{m}\sum\limits_{k=0}^{m}%
\sum\limits_{j=m}^{\infty
}(-1)^{j-m}\frac{\Gamma(-\tilde{\alpha} +1)}{(j-m)!\Gamma(-\tilde{\alpha} -j+m+1)}a_{nj}\text{;}
\end{equation*}%
also $\beta =\lim_{n\rightarrow \infty }( \sum_{k=0}^{\infty }%
\hat{a}_{nk}-\gamma _{n})$. We also write $\hat{B}=(\hat{b}_{nk})_{n,k=0}^{\infty }$ for the matrix with $\hat{b}_{nk}=\hat{a}_{nk}-\hat{\alpha}_{k}$ for each
$n,k\in \mathbb{N}_{0}$ and $\delta =(\delta _{n})_{n=0}^{\infty }$
for the sequence with $ \delta _{n}=\sum_{k=0}^{\infty }\hat{\alpha}_{k}-\gamma
_{n}+\beta \text{ }(n=0,1,\ldots ) $

\end{Thm}

\begin{proof}

The conditions in {\bf 1.} and {\bf 2.} are immediate consequence of
(\cite{djolovıc}, Corollary 3.6(a)). We define $P_{r}:\ell _{\infty
}\rightarrow \ell _{\infty }$ by $P_{r}(x)=x^{[r]}$ for all $x \in
\ell _{\infty }$ and $r=0,1,\dots$, $R_{r}=I-P_{r}$, and write
$L=L_{A}$ and $\bar{B}=\bar{B}_{\ell _{\infty }}$ for short. Then it
follows from (\ref{compactoperator}), (\cite{mara}, Theorem 2.12)
and Lemma \ref{operatornorm.2}(i) that

\begin{eqnarray*}
0 &\leq &\left\Vert L\right\Vert _{\chi }=\chi (L(\bar{B})) \\
&\leq &\chi (P_{r}(L(\bar{B})))+\chi (R_{r}(L(\bar{B}))) \\
&=&\chi (R_{r}(L(\bar{B})))\leq \sup_{x\in \bar{B}}\left\Vert
R_{r}(L(x))\right\Vert _{\infty }=\left\Vert
\hat{A}^{[r]}\right\Vert _{(X,\infty )}.
\end{eqnarray*}

Then, {\bf 3.} holds.

The conditions in {\bf 4.} and {\bf 6.} are immediate consequence of
(\cite{djolovıc}, Theorem 3.7 (b), (a)). Part {\bf 5.} follows by a
similar argument as part {\bf 3.}; we use Lemma \ref{operatornorm.3}(i)
instead of Lemma \ref{operatornorm.2}(i).
\end{proof}

\begin{Cor}
Let $X$ be one of the spaces $c_{0}(\Delta ^{(  \tilde{\alpha} )} )$ or\
$\ell_{\infty} (\Delta ^{(  \widetilde{\alpha} )}  )$. We obtain as an immediate consequence
of (\ref{c14}) and Theorem \ref{generaltheorem} ((1.1), (3.1) and
(5.1)).

\begin{itemize}
\item[(i) ] If $A\in (X,c_{0})$, then $L_{A}$ is compact if and only if
\begin{equation}
\lim\limits_{r\rightarrow \infty }\left( \sup\limits_{n\geq r}\left(
\sum\limits_{k=0}^{\infty}\left\vert \hat{a}_{nk}\right\vert \right) \right) =0.
\label{comp2}
\end{equation}

\item[(ii)] If $A\in (X,c)$, then $L_{A}$ is compact if and only if
\begin{equation*}
\lim\limits_{r\rightarrow \infty }\sup\limits_{n\geq r}\left(
\sum\limits_{k=0}^{\infty}\left\vert \hat{a}_{nk}-%
\hat{\alpha}_{k}\right\vert \right) =0.
\end{equation*}

\item[(iii)] If $A\in (X,\ell _{\infty })$, then $L_{A}$ is
compact if the condition\ (\ref{comp2}) holds.
\end{itemize}
\end{Cor}

\begin{Cor}
We obtain as an immediate consequence of (\ref{c14}) and Theorem
\ref{generaltheorem} ((2.1), (4.1) and (6.1)).

\begin{itemize}
\item[(i)] If $A\in (c(\Delta ^{(  \widetilde{\alpha} )} ),c_{0})$, then $L_{A}$
is compact if and only if
\begin{equation}
\lim\limits_{r\rightarrow \infty }\left( \sup\limits_{n\geq r}\left(
\sum\limits_{k=0}^{\infty}\left\vert \hat{a}_{nk}\right\vert +\left\vert \gamma
_{n}\right\vert \right) \right) =0. \label{comp3}
\end{equation}

\item[(ii)] If $A\in (c(\Delta ^{(  \widetilde{\alpha} )} ),c)$, then $L_{A}$ is
compact if and only if
\begin{equation*}
\lim\limits_{r\rightarrow \infty }\sup\limits_{n\geq r}\left(
\sum\limits_{k=0}^{\infty}\left\vert \hat{a}_{nk}-%
\hat{\alpha}_{k}\right\vert +\left\vert \sum\limits_{k=0}^{\infty}\hat{\alpha}%
_{k}-\gamma _{n}-\beta \right\vert \right) =0.
\end{equation*}

\item[(iii)] If $A\in (c(\Delta ^{(  \widetilde{\alpha} )} ),\ell _{\infty })$,
then $L_{A}$ is compact if (\ref{comp3}) holds.
\end{itemize}
\end{Cor}

\begin{Cor}
We obtain as an immediate consequence of (\ref{c14}), Theorem
\ref{generaltheorem} ((7.1), (8.1)).

\begin{itemize}
\item[(i)] If $A\in (c_{0}(\Delta ^{(  \tilde{\alpha} )} ),\ell _{1})$ or $A\in
(\ell_{\infty}(\Delta ^{(  \tilde{\alpha} )} ),\ell _{1})$, then $L_{A}$ is compact if
and only if
\begin{equation*}
\lim\limits_{r\rightarrow \infty }\left( \sup\limits_{\substack{
N\subset \mathbb{N}_{0} \\ \text{finite}}}\left\Vert
\sum\limits_{n\in N_{r}}\sum\limits_{j=k}^{\infty
}(-1)^{j-k}\frac{\Gamma(-\tilde{\alpha} +1)}{(j-k)!\Gamma(-\tilde{\alpha} -j+k+1)}a_{nj}\right\Vert _{1}\right) =0.
\end{equation*}

\item[(ii)] If $A\in (c(\Delta ^{(  \tilde{\alpha} )} ),\ell _{1})$, then $%
L_{A}$ is compact if and only if
\begin{equation*}
\lim\limits_{r\rightarrow \infty }\sup\limits_{\substack{ N\subset
\mathbb{N}_{0} \\ \text{finite}}}\left(
\sum\limits_{k=0}^{\infty}\left\vert \sum\limits_{n\in N_{r}}\sum\limits_{j=k}^{\infty
}(-1)^{j-k}\frac{\Gamma(-\tilde{\alpha} +1)}{(j-k)!\Gamma(-\tilde{\alpha} -j+k+1)}a_{nj}\right\vert +\left\vert \sum\limits_{n\in
N_{r}}\gamma _{n}\right\vert \right) =0.
\end{equation*}
\end{itemize}
\end{Cor}
\section*{Compliance with ethical standards}
\subsection*{Conflict of interest}
The authors declare that they have no conflict of
interest.

\subsection*{Ethical approval}
This article does not contain any studies with human
participants or animals performed by any of the authors.

\end{document}